\documentclass[11pt]{amsart}
\usepackage{amsmath,amssymb,amsthm,amscd}
\usepackage{pbox}
\usepackage{array}
\usepackage{bbm}
\usepackage{comment}
\usepackage{xfrac}
\usepackage{color}

\usepackage{grcalck}
\usepackage{grcalcx}
\usepackage{grcawol}

\newtheorem{thm}{Theorem}[section]

\newtheorem{lem}[thm]{Lemma}

\theoremstyle{definition}
\newtheorem{defn}[thm]{Definition}
\newtheorem{rmk}[thm]{Remark}

\newcommand{\C}{{\mathcal C}}

\newcommand\Rep{\operatorname{Rep}}
\newcommand{\Irr}{\operatorname{Irr}}
\newcommand\Vect{\operatorname{Vec}}

\newcommand\Hom{\operatorname{Hom}}

\newcommand\Id{\operatorname{Id}}

\newcommand\Mod{\operatorname{Mod}}
\newcommand\Bimod{\operatorname{Bimod}}

\newcommand\F{\mathbb{F}}

\newcommand\Opext{\operatorname{Opext}}

\newcommand\FPdim{\operatorname{FPdim}}

\newcommand\Tr{\operatorname{Tr}}

\title[Noncommutative near-group categories]{Algebraic realization of noncommutative near-group fusion categories}

\author{Masaki Izumi}
\address{Graduate School of Science, Kyoto University, Kitashirakawa Oiwake-cho, Sakyo-ku, Kyoto 606-8502, Japan}
\email{izumi@math.kyoto-u.ac.jp}

\author{Henry Tucker}
\address{Dept. of Mathematics, University of California, San Diego, 9500 Gilman Drive \#0112, La Jolla, CA 92093, USA}
\email{hjtucker@ucsd.edu}



\date{\today}

\begin{document}

\maketitle
\begin{abstract} Noncommutative near-group fusion categories were completely classified in the previous work of the first named author by using an operator algebraic method (and hence under the assumption of unitarity), and they were shown to be group theoretical though the corresponding pointed categories were not identified. In this note we give a purely algebraic construction of the noncommutative near-group fusion categories starting from pointed categories categorically Morita equivalent to them. 
\end{abstract}



\section{Introduction}
Fusion categories are considered to be generalization of the category $\Rep(G)$ of finite dimensional representations of finite groups $G$. 
As  we can associate the representation ring $R(G)$ with $\Rep(G)$, 
given a fusion category $\C$ we can define the Grothendieck ring $K_0(\C)$, which expresses the fusion rules of the tensor product in $\C$. 
That is to say, $K_0(\C)$ is  the $\mathbb{Z}$-based ring with basis given by $\Irr(\C)$, the set of isomorphism classes of simple objects in $\C$ 
with multiplication and addition given by the tensor product and direct sum operations, respectively. 
A natural classification question asks: what fusion categories have a given ring $R$ as their Grothendieck ring, i.e. 
which $\C$ have $K_0(\C) \cong R$ as $\mathbb{Z}$-based rings? 
The answer in the case of a group ring $\mathbb{Z}G$ for a finite group $G$ has a well-known answer due to Eilenberg and Mac Lane: 
if $\C$ is a fusion category with $K_0(\C)\cong \mathbb{Z}G$ then $\C$ is monoidally equivalent to a category $\Vect_G^\omega$ of 
$G$-graded complex vector spaces with associativity natural isomorphisms given by some 3-cocycle $\omega \in Z^3(G,\mathbb{C}^\times)$. 
These are the pointed fusion categories. 
Note that the 3-cocycle condition is equivalent to the pentagon axiom for these monoidal categories; two such categories are equivalent 
if and only if the groups are isomorphic and the 3-cocycles are cohomologous in $H^3(G, \mathbb{C}^\times)$.

Further answers to this question have focused primarily on the near-group categories formally introduced in \cite{S03}. 
These are fusion categories with exactly one object which is non-invertible under the tensor product. 
More precisely:

\begin{defn} A near-group category $\C$ with a finite group $G$ is a fusion category with the the Grothendieck ring 
$K_0(\C) = \mathbb{Z}[G\cup \{\rho\}]$ where multiplication is given by the group law of $G$ and:
\[
\rho g = \rho = g \rho \quad \textrm{and} \quad \rho^2 = m\rho + \sum_{g\in G} g.
\]
We denote this ring by $NG(G,m)$. 
The number $m$ is called the multiplicity of $\C$. 
\end{defn}

The first major classification of categories of this type was by Tambara-Yamagami \cite{ty} for those with $m=0$. 
By solving directly the pentagon equations arising from the associativity law they showed that the near-group categories with $G$ and multiplicity $m=0$ are completely classified by the data $(b,\epsilon)$ where $b$ is a non-degenerate symmetric bicharacter of $G$, and $\epsilon \in \{1,-1\}$ corresponding to the second categorical Frobenius-Schur indicator of $\rho$. 
The group $G$ must be abelian in this case. 

The features of near-group categories with $m>0$ are considerably different depending on whether the Frobenius-Perron dimension $\FPdim \rho$ of $\rho$ is an integer or not (see \cite{I17} and the arXiv version of \cite{O15}). 
The integral case with abelian $G$ is the easiest among the others, and the classification was already done in \cite{ENO04}. 
Such a category is categorically Morita equivalent to a pointed category with the affine group $\F_q\rtimes \F_q^\times$, and $G$ is identified with the dual group of $\F_q^\times$, isomorphic to the cyclic group of order $q-1$. 
The variety coming from the third cohomology was also precisely determined in \cite{ENO04}.  

In the irrational case, the group $G$ is always abelian and the number $m$ is a multiple of $|G|$. 
Under the assumption of unitarity, the classification problem in this case is completely reduced to solving the polynomial equations 
obtained by the first named author by using the Cuntz algebras (see \cite{EG14}, \cite{I93}, \cite{I01II}, \cite{I17}). 
Their Drinfeld centers were also computed in \cite{EG14} and \cite{I01II} applying the tube algebra method presented in \cite{I01I}. 

The remaining case is when $G$ is noncommutative, and somehow this case was overlooked in \cite{S03}. 
The following theorem was proved by the first named author with operator algebra technique:

\begin{thm}[{\cite[Theorem 6.1]{I17}}]
Let $G$ be a non-abelian finite group. 
Then a unitary near-group category $\C$ with $K_0(\C)\cong NG(G,m)$ exists if and only if $G$ is an extra-special $2$-group. 
In particular $|G| = 2^{2n+1}$, $m = 2^n$, and $\FPdim(\rho) = 2^{n+1}$. 

Furthermore, for each extra-special $2$-group there exists exactly three different unitary near-group categories, and they are distinguished 
by the third Frobenius-Schur indicators of $\rho$.
\end{thm}

Recall that a finite 2-group $G$ is said to be an extra-special if the center and the derived subgroup of $G$ coincide and have order 2. 
It is known that the order of every extra-special 2-group is an odd power of 2. 
Note that for $n=1$ above the group $G$ is either the dihedral group $D_8$ or the quaternion group $Q_8$. 
Furthermore, it is known that there are exactly two non-isomorphic extra-special $2$-groups $G$ for each order $|G|=2^{2n+1}$, 
and they are obtained as central products of $D_8$ or $Q_8$ (see \cite[Exercise 5.3.7,(i)]{R93}).

In \cite[Corollary 6.14]{I17} it was also shown that the noncommutative near-group categories are group theoretical, that is, 
categorically Morita equivalent to pointed categories. 
Summing up these results, we see that for a given odd natural number $2n+1$, there exist exactly 6 unitary noncommutative near-group categories 
with global dimension $2^{2n+1}\cdot 3$, and hence the corresponding pointed categories are given by groups of order $2^{2n+1}\cdot 3$. 
It was conjectured in \cite[Example 6.16]{I17} that all the 6 pointed categories have the group $\F_2^{2n}\rtimes SL(2,\F_2)$ with 6 different third cohomology classes 
coming from $H^3(S_3,\mathbb{C}^\times)\cong \mathbb{Z}_6$ (recall that $SL(2,\F_2)$ is isomorphic to the symmetric group $S_3$). 
The purpose of this note is to verify this conjecture (Theorem \ref{basic}, Theorem \ref{general}). 

To investigate properties of group theoretical categories arising from the group $\F_2^{2n}\rtimes SL(2,\F_2)$ together with 
appropriate choices of 3-cocycles and subgroups, we employ Schauenburg's formula \cite{S15} for the higher Frobenius-Schur indicators 
and Gelaki and Naidu's description \cite{GN09} of the group of invertible objects. 
Our computation explicitly verifies a Frobenius-Schur indicator rigidity phenomenon for the noncommutative near-group fusion rules: 
each inequivalent category with the given fusion rules must have distinct Frobenius-Schur data. 

\subsection{Acknowledgments}
This project was made possible by NSF East Asia and Pacific Summer Institutes (EAPSI) award \#1613812 titled ``Classification of Fusion Categories with one Non-Invertible Object'' supported jointly by NSF and Japan Society for Promotion of Science (JSPS). Both authors wish to thank Richard Ng for many useful discussions. 
The first named author is supported in part by JSPS KAKENHI Grant Number JP15H03623. 

\section{Group-theoretical fusion categories}

For our purposes {\bf tensor categories} are defined to be $\mathbb{C}$-linear abelian monoidal categories. We refer the reader to \cite{EGNO15} for a comprehensive overview of the theory of tensor categories. {\bf Fusion categories} are tensor categories that are:
\begin{itemize}
    \item {\it rigid}, that is, objects have duals, and
    \item {\it semisimple} with finitely many simple objects and simple unit object.
\end{itemize}
The rigidity morphisms $\textrm{ev}_V: V^*\otimes V \to \mathbbm{1}$ and $\textrm{coev}_V: \mathbbm{1} \to V \otimes V^*$ are given by the diagrams:

\[
\gbeg33
\got1{V^{*}}\gvac1\got1V\gnl
\gwev3\gnl
\gend
\qquad
\gbeg33
\gnl
\gwdb3\gnl
\gob1 V \gvac1 \gob1{V^{*}}
\gend
\]
satisfying the relations:
\[
        \gbeg64
        \gvac4\got{1}{V} \gnl
        \gwdb3 \gvac1 \gcl1 \gnl
        \gcl1 \gvac1 \gwev3 \gnl
        \gob{1}{V}
        \gend =
        \gbeg14
        \got{1}{V}\gnl
        \gcl2\gnl
        \gob{1}{V}
        \gend
        \qquad \textrm{and} \qquad
        \gbeg64
        \got{1}{V^{*}} \gvac4 \gnl
        \gcl1 \gvac1 \gwdb3 \gnl
        \gwev3 \gvac1 \gcl1 \gnl
        \gvac4 \gob{1}{V^{*}}
        \gend =
        \gbeg14
        \got{1}{V^{*}}\gnl
        \gcl2\gnl
        \gob{1}{V^{*}}
        \gend
\]

\subsection{Categorical Morita equivalence}

Let $\C$ be a fusion category. A {\bf module category} $\mathcal{M}$ over $\C$ is the categorical analog of a module over the ring $K_0(\C)$. 
That is, a linear abelian category with an action given by a bifunctor $\otimes_\mathcal{M}: \C \times \mathcal{M} \to \mathcal{M}$ equipped 
with a module associativity natural isomorphism $m: (\cdot \otimes_\C \cdot) \otimes_\mathcal{M} \cdot \to \cdot \otimes_\C ( \cdot \otimes_\mathcal{M} \cdot)$ 
which satisfies the usual pentagon relation. 
There are likewise analogous definitions for indecomposable module categories and $\C$-module functors. 
(See \cite[Chap. 7]{EGNO15} for further details and definitions.) 
It was shown in \cite{O01} that module categories $\mathcal{M}$ over a fusion category $\C$ are characterized by (associative) {\bf algebra objects} in $\C$. $A\in \C$ is an algebra object if it has multiplication $\mu:A\otimes A \to A$ and unit $\eta: \mathbf{1} \to A$ morphisms in the category satisfying some compatibility conditions with the monoidal structure $\otimes$ of $\C$. Similarly one may look at the category $\Mod_\C(A)$ of {\bf $A$-module objects} in $\C$, and Ostrik's result tells us that, for a given $\C$-module category $\mathcal{M}$ we have $\mathcal{M}\simeq\Mod_\C(A)$ as abelian categories for some algebra object $A\in \C$.

Given an indecomposable module category $\mathcal{M}=\Mod_\C(A)$ over $\C$ we can form the categorical {\bf Morita dual} tensor category 
$\C_\mathcal{M}^*:=\mathcal{E}nd_\C(\mathcal{M})$, that is the category of $\C$-module endofunctors of $\mathcal{M}$. 
Equivalently, the dual arises from the category of {\bf $(A,A)$-bimodule objects} in $\C$:
\[
\C^*_\mathcal{M} \simeq \Bimod_\C(A)
\]
We say that the two fusion categories $\C$ and $\C_\mathcal{M}^*$ are {\bf categorically Morita equivalent}. 

A fusion category $\C$ is said to be {\bf group-theoretical} if it is categorically Morita equivalent to a pointed fusion category. 
Given a subgroup $H\leq G$ and a 2-cochain $\psi\in C^2(H,\mathbb{C}^\times)$ such that $d\psi = \omega|_H$ we may form the {\bf twisted group algebra} $\mathbb{C}^\psi H$. The choice of $\psi$ makes this an algebra object in $\Vect_G^\omega$, and in fact all module categories $\mathcal{M}$ over $\Vect_G^\omega$ are of the form $\Mod_{\Vect_G^\omega}(\mathbb{C}^\psi H)$ for some pair $(H,\psi)$ as above. Hence any group-theoretical category is of the following form:
\[
\C(G,\omega,H,\psi):= \Bimod_{\Vect^G_\omega}(\mathbb{C}^\psi H)
\]
For example, $\C(G,1,G,1) \simeq \Rep(G)$ and $\C(G,\omega, \{e\},1)\simeq \Vect_G^\omega $

For the remainder of this note we shall denote the adjoint actions of $G$ on itself by  
$x\lhd g:= g^{-1}xg$ for all $x,g\in G$. 
Furthermore, for $f\in C^k(G,\mathbb{C}^\times)$ we define:
\[
^g f(x_1, \ldots, x_k) := f(x_1\lhd g, \ldots, x_k\lhd g). 
\]

The simple objects of $\C(G,\omega,H,\psi)$ were found in \cite{O02} (see also \cite[Example 9.7.4]{EGNO15}) 
to be parameterized by pairs $(HgH, \chi)$ where:
\begin{itemize}
\item $HgH \in H\backslash G/H$ is a double coset with representative $g\in G$, and
\item $\chi$ is an irreducible projective character of the group $H^g:=H\cap gHg^{-1}$ with Schur multiplier $\psi^g\in Z^2(H^g,\mathbb{C}^\times)$, a well-defined $2$-cocycle defined by:
\[
\psi^g(h,k) := 
\psi(h,k)\psi(h^{-1}\lhd g, k^{-1}\lhd g) 
\frac{  \omega(h, k, g) \omega(h, kg, k^{-1}\lhd g )  }{  \omega(hkg, k^{-1}\lhd g, h^{-1}\lhd g) }.
\]
\end{itemize}
We denote by $X_{g,\chi}$ the corresponding simple object. 
Due to an important theorem of Natale in \cite{N05}, given a group-theoretical fusion category $\C(G,\omega,H,\psi)$ 
there exists a tensor equivalent one $\C(G,\tilde{\omega},H,1)$ with trivial $2$-cochain $\psi$. 
Moreover, $\tilde{\omega}$ can be chosen to be \textbf{adapted} in the sense that it is trivial on $G\times G\times H$. 
So for any group-theoretical fusion category we have:
\[
\C(G,\omega,H,\psi)\simeq_\otimes \C(G,\tilde{\omega},H,1)
\]
In this case, the 2-cocycle $\psi^g(h,k)$ takes the form $\tilde{\omega}(h,k,g)$, which we denote 
by $\tilde{\omega}_g(h,k)$.

\subsection{Frobenius-Schur indicators}\label{FS}

The group theoretical fusion categories are {\bf pivotal} fusion categories: there is a natural isomorphism $j:(-)^{**} \to \Id_\C$ from the double dual endofunctor on $\C$ to the identity endofunctor. Define $X^{\otimes n}$ to be the $n$-fold tensor product of $X$ with all parentheses to the right. Now we have the following invariant of pivotal categories:

\begin{defn}[{\cite{NS07}}] 
 Given an object $X$ in a pivotal fusion category $\C$ we define the linear map $E_X^{(k)}$ on the finite-dimensional vector space $\Hom(\,\mathbbm{1},X^{\otimes k})$ as follows: 
\[
E^{(k)}_X: \underbrace{\gbeg43\glmpb\gdnot f\gcmpb\gcmpb\grmpb\gnl
  \gcl1\gcl1\gnot{\cdots}\gvac1\gcl1\gnl
  \gob1X\gob1X\gob1\cdots\gob1X\gend}_k
  \mapsto
  \gbeg95\gwdb8\gnl
  \gcl1\glmpb\gdnot f\gcmpb\gcmpb\grmpb\gvac2\gcl1\gnl
  \gev\gcl2\gnot\cdots\gvac1\gcl2\gvac1\glmp\gnot{\mbox{ \tiny $j^{-1}_X$} }\gcmptb\grmp\gnl
  \gvac7 \gcl1\gnl
  \gvac2\gob1X\gob1\cdots\gob1X\gvac2\gob1X\gend
\]
Then the {\bf $k^{\text{th}}$ Frobenius-Schur indicators} are given by the (usual) trace:
\[
\nu_k(X)=\Tr(E^{(k)}_X)
\]
\end{defn}
\noindent The FS indicators are an invariant of tensor categories: if $F:\C \to \mathcal{D}$ is an equivalence of tensor categories then $\nu_k(X\in \C) = \nu_k(F(X)\in \mathcal{D})$ for all positive integers $k$ and objects $X\in \C$.

We will employ the FS indicators to show that our constructions of the noncommutative near-group categories are distinct. 
We make use of the following formula of Schauenburg for the FS indicators of group-theoretical categories which have an adapted cocycle.

\begin{thm}[{\cite[Theorem 1]{S15}}] \label{fs-gt-form}
Suppose $\omega \in Z^3(G,\mathbb{C}^\times)$ is adapted cocycle for the subgroup $H$. 
Let $X_{g,\chi}\in \C(G,\omega,H,1)$ be the simple object associated with the pair $(HgH,\chi)$. 
Then:
\[
\nu_k(X_{g,\chi}) = \frac{1}{|H^g|}\sum_{\substack{r\in gH\\ r^k\in H^g}} \pi_{-k}(r) \chi(r^{-k}),
\]
where $\pi_k(x)$ is a recursively defined scalar for $x\in G$ given by $\pi_0(x) = 1$ and:
\[
\pi_{k+1}(x) = \omega(x,x^k,x)\pi_k(x).
\]
\end{thm}

\subsection{Group of invertible objects} 
The group $\Gamma(\C)$ of the invertible objects in a group theoretical category $\C$  was computed in 
\cite[Theorem 5.2]{GN09}. 
We recall their result in the case $\C = \C(G, \omega, H, 1)$ with an adapted cocycle $\omega$. 

Let $R$ be a set of representatives of the double cosets $H\backslash G / H$. Then we define a group $K$ isomorphic to a subgroup of $N_G(H)/H$ as follows. 
Let 
\[
K=\{g \in R \; | \; g\in N_G(H) \; \textrm{and} \; [\omega_g]=1\in H^2(H,\mathbb{C}^\times) \}
\]
as a set. 
Then for any $k_1,k_2\in K$, there exists a unique $k_3\in K$ with $k_1k_2\in k_3H$, and we define 
$k_1\cdot k_2=k_3$. 
Let $\widehat{H}=\Hom(H,\mathbb{C}^\times)$. 
Then $K$ acts on $\widehat{H}$ as ${}^k\chi(h)=\chi(h\lhd k)$. 
We define a 2-cocycle $\nu\in Z^2(K,\widehat{H})$ as follows: 
\begin{itemize}
\item Since $\omega_k$ is a coboundary for every $k\in K$, there exists $\eta_k \in C^1(H,\mathbb{C}^\times)$ satisfying 
$d\eta_k =\omega_k$. 
We set $\eta_e =1$ for $k=e$. 
\item We define $\nu \in Z^2(K,\hat{H})$ by
\[
\nu(s,t) = \frac{\eta_s (^s\eta_t)}{\eta_{s\cdot t}}, 
\]
where $s, t \in K$.
\end{itemize}

\begin{thm}[{\cite[Theorem 5.2]{GN09}}] \label{Inv} Let the notation be as above. 
Then the group $\Gamma(\C)$ fits into the following exact sequence:
\[
1 \to \widehat{H} \to \Gamma(\C) \to K \to 1
\]
with a 2-cocycle $\nu$. 
More precisely, the group $\Gamma(\C)$ is identified with the set $\widehat{H}\times K$ 
with multiplication given by:
\[
(\chi,s)\cdot (\psi,t) := (\nu(s,t)\chi \;{}^s\psi,s\cdot t)
\]
where $s,t \in K$, $\chi, \psi \in \widehat{H}$. 
\end{thm}


\section{The basic case}
\subsection{Cohomology for $S_3$} The cohomology $H^3(S_3,\mathbb{C}^\times)$ of the symmetric group $S_3$ plays a key r\^ole in the sequel. 
We provide here explicit formulae for the $3$-cocycles as computed in \cite[\S 6.3 (6.19)-(6.20)]{CGR00} (with a modification for our purpose). 
The symmetric group $S_3$ is a semi-direct product $\mathbb{Z}_3\rtimes_{-1} \mathbb{Z}_2$, 
and we uniquely parameterize group elements as $(123)^{x}(13)^y$, $x\in \{0,1,2\}$, $y\in \{0,1\}$. 
Using this convention, the group law can be written as 
\[
(x_1,y_1)\cdot(x_2,y_2) = ( [x_1+(-1)^{y_1}x_2]_3,[y_1+y_2]_2 ), 
\]
where $[-]_p$ denotes reduction modulo $p$. 

Let $\alpha:\mathbb{Z}^3\to \mathbb{C}^\times$ be a function defined by  
\[\alpha(x_1,x_2,x_3)=\exp(\frac{2\pi i(x_1+x_2-[x_1+x_2]_3)x_3}{9}).\]
Although it is a fact that the restriction of $\alpha$ to $\{0,1,2\}^3$ gives a generator of $H^3(\mathbb{Z}_3,\mathbb{C}^\times)\cong \mathbb{Z}_3$, 
we emphasize here that we are using $\alpha$ defined on $\mathbb{Z}^3$ in the below.  

We can write a representative of a generator of $H^3(S_3, \mathbb{C}^\times)\cong\mathbb{Z}_6$ as follows:
\[\omega_0((x_1,y_1),(x_2,y_2),(x_3,y_3)) =\alpha(x_1,(-1)^{y_1}x_2,(-1)^{y_1+y_2}x_3)(-1)^{y_1y_2y_3}.\]
Indeed, we can directly check the cocycle relation as 
\begin{align*}
\lefteqn{d\omega_0((x_0,y_0),(x_1,y_1),(x_2,y_2),(x_3,y_3))} \\
 &=\alpha(x_1,(-1)^{y_1}x_2,(-1)^{y_1+y_2}x_3)\\
 &\times \alpha([x_0+(-1)^{y_0}x_1]_3,(-1)^{y_0+y_1}x_2,(-1)^{y_0+y_1+y_2}x_3)^{-1}\\ 
 &\times \alpha(x_0,(-1)^{y_0}[x_1+(-1)^{y_1}x_2]_3,(-1)^{y_0+y_1+y_2}x_3) \\
 &\times \alpha(x_0,(-1)^{y_0}x_1,(-1)^{y_0+y_1}[x_2+(-1)^{y_2}x_3]_3)^{-1} \\
 &\times \alpha(x_0,(-1)^{y_0}x_1,(-1)^{y_0+y_1}x_2)\\
 &\times (-1)^{y_1y_2y_3}(-1)^{(y_0+y_1)y_2y_3}(-1)^{y_0(y_1+y_2)y_3}(-1)^{y_0y_1(y_2+y_3)}(-1)^{y_0y_1y_2}\\
 &=\exp\frac{2\pi i c}{9},
\end{align*}
with 
\begin{align*}
\lefteqn{c} \\
 &=(x_1+(-1)^{y_1}x_2-[x_1+(-1)^{y_1}x_2]_3)(-1)^{y_1+y_2}x_3\\
 &-([x_0+(-1)^{y_0}x_1]_3+(-1)^{y_0+y_1}x_2-[[x_0+(-1)^{y_0}x_1]_3+(-1)^{y_0+y_1}x_2]_3)\\
 &\times (-1)^{y_0+y_1+y_2}x_3 \\
 &+ (x_0+(-1)^{y_0}[x_1+(-1)^{y_1}x_2]_3-[x_0+(-1)^{y_0}[x_1+(-1)^{y_1}x_2]_3]_3)\\
 &\times (-1)^{y_0+y_1+y_2}x_3\\
 &-(x_0+(-1)^{y_0}x_1-[x_0+(-1)^{y_0}x_1]_3)(-1)^{y_0+y_1}[x_2+(-1)^{y_2}x_3]_3 \\
 &+(x_0+(-1)^{y_0}x_1-[x_0+(-1)^{y_0}x_1]_3)(-1)^{y_0+y_1}x_2\\
 &=(x_0+(-1)^{y_0}x_1-[x_0+(-1)^{y_0}x_1]_3)\\
 &\times (x_2+(-1)^{y_2}x_3-[x_2+(-1)^{y_2}x_3]_3)(-1)^{y_0+y_1}\\
 &\equiv 0\mod 9. 
\end{align*}
Since the restrictions of $\omega_0$ to the Sylow subgroups $\langle(123)\rangle$ and $\langle(13)\rangle$ respectively 
generate their third cohomology groups, we see that the class $[\omega_0]$ generates $H^3(S_3,\mathbb{C}^\times)\cong \mathbb{Z}_6$.  

We note that $\omega_0$ satisfies $\omega_0((x_1,y_1),(x_2,y_2),(0,y_3))=(-1)^{y_1y_2y_3}$. 

\subsection{Inflation of $3$-cocycles on $S_3$}
We naturally identify $S_3$ as a subgroup of $S_4$. 
Let $N=\{e,(12)(34),(13)(24),(14)(23)\}\cong \mathbb{Z}_2\times \mathbb{Z}_2$. 
Then $S_4$ is a semi-direct product $S_4=N\rtimes S_3$, and we denote by $\pi$ the quotient map 
$$\pi:S_4\to S_4/N\cong S_3,$$
whose restriction to $S_3$ is the identity map.  

We apply the cohomology functor $H^3(\cdot, \mathbb{C}^\times)$ to obtain the {\bf inflation map}:
\[
\inf: H^3(S_3, \mathbb{C}^\times) \to H^3(S_4, \mathbb{C}^\times)
\]
which is given by precomposition by the quotient map $\pi:S_4 \to S_3$:
\[
\inf(\omega)(x,y,z) = \omega(\pi(x),\pi(y),\pi(z)).
\]
Since $(1234)=(14)(23)(13)$, we have $\pi((1234))=(13)$. 
We can uniquely parameterize elements of $S_4$ as $(123)^x(13)^y(1234)^z$ with $x\in \{0,1,2\}$, $y\in \{0,1\}$, and $z\in \{0,1,2,3\}$. 
Then by definition
\begin{align*}
\lefteqn{\inf \omega_0((x_1,y_1,z_1),(x_2,y_2,z_2),(x_3,y_3,z_3))} \\
 &=\omega_0((x_1,[y_1+z_1]_2),(x_2,[y_2+z_2]_2),(x_3,[y_3+z_3]_2))\\
 &=\alpha(x_1,(-1)^{y_1+z_1}x_2,(-1)^{y_1+y_2+z_1+z_2}x_3)(-1)^{(y_1+z_1)(y_2+z_2)(y_3+z_3)}.
\end{align*}

\subsection{Construction of adapted cocycle}
Let $H=\langle(1234)\rangle\cong \mathbb{Z}_4$. 
Then we have $S_4=S_3\cdot H$ and $S_3\cap H=\{e\}$, that is, $(S_3,H)$ is a matched pair. 
We look for an adapted 3-cocycle of $S_4$ for $H$ cohomologous to $\operatorname{inf}\omega_0$. 

We define a function $\epsilon :S_4\to \{0,1\}$ by $\operatorname{sgn}(g)=(-1)^{\epsilon(g)}$, where 
$\operatorname{sgn} g$ is the signature of $g$. 
By definition we have $\epsilon(g_1g_2)\equiv \epsilon(g_1)+\epsilon(g_2)\mod 2$. 
Then we get 
\[\inf \omega_0(g_1,g_2,h)=(-1)^{\epsilon (g_1)\epsilon(g_2)\epsilon(h)} \]
for all $g_1,g_2\in S_4$ and $h\in H$. 
We need to find a 2-cochain $\xi\in C^2(S_4,\mathbb{C}^\times)$ satisfying 
$$(-1)^{\epsilon (g_1)\epsilon(g_2)\epsilon(h)}=d\xi(g_1,g_2,h),$$
for all $g_1,g_2\in S_4$ and $h\in H$.

Let $f_0:H\to \{1,-1\}$ be a function given by 
$$f_0((1234)^z)=\left\{
\begin{array}{ll}
1 , &\quad z=0,1  \\
-1 , &\quad z=2,3.
\end{array}
\right.
$$
Then $f_0$ satisfies $f_0(h_1h_2)=f_0(h_1)f_0(h_2)(-1)^{\epsilon(h_1)\epsilon(h_2)}$. 
Since $(S_3,H)$ is a matched pair, we can extend $f_0$ to $f:S_4\to \{1,-1\}$ by setting 
$$f(\sigma h)=f_0(h)(-1)^{\epsilon(\sigma)\epsilon(h)},$$ 
for all $\sigma\in S_3$ and $h\in H$. 
Then $f$ satisfies $f(gh)=f(g)f(h)(-1)^{\epsilon(g)\epsilon(h)}$ for all $g\in S_4$ and $h\in H$. 
We set $\xi(g_1,g_2)=f(g_2)^{\epsilon(g_1)}$ . 

\begin{lem} Let $g_i=\sigma_ih_i$ for $i=1,2,3$ with $\sigma_i\in S_3$ and $h_i\in H$. 
Then \begin{align*}
\lefteqn{d \xi(g_1,g_2,g_3)} \\
 &=(f(g_2\sigma_3)f(g_2))^{\epsilon(g_1)}(-1)^{\epsilon(g_1)\epsilon(g_2)\epsilon(h_3)} \\
 &=(f(g_2\sigma_3)f(g_2)(-1)^{\epsilon(g_2)\epsilon(\sigma_3)})^{\epsilon(g_1)}.
\end{align*}
In particular, if $g_3=h\in H$, we have 
$$d \xi (g_1,g_2,h)=(-1)^{\epsilon(g_1)\epsilon(g_2)\epsilon(h)}.$$
\end{lem}

\begin{proof} By construction, we have 
\begin{align*}
\lefteqn{d \xi(g_1,g_2,g_3)} \\
 &=\xi(g_2,g_3)\xi(g_1g_2,g_3)^{-1}\xi(g_1,g_2g_3)\xi(g_1,g_2)^{-1} \\
 &=f(g_3)^{\epsilon(g_2)-\epsilon(g_1g_2)}f(g_2g_3)^{\epsilon(g_1)}f(g_2)^{-\epsilon(g_1)}\\
 &=(f(g_2\sigma_3h_3)f(g_2)f(\sigma_3h_3))^{\epsilon(g_1)}\\
 &=(f(g_2\sigma_3)f(h_3)(-1)^{\epsilon(g_2\sigma_3)\epsilon(h_3)}f(g_2)f(h_3)(-1)^{\epsilon(\sigma_3)\epsilon(h_3)})^{\epsilon(g_1)}\\
 &=(f(g_2\sigma_3)f(g_2))^{\epsilon(g_1)}(-1)^{\epsilon(g_1)\epsilon(g_2)\epsilon(h_3)}.
\end{align*}
\end{proof}

We define an adapted cocycle $\omega\in Z^3(S_4,\mathbb{C}^{\times})$ by 
$$\omega(g_1,g_2,g_3)=\inf \omega_0(g_1,g_2,g_3)d \xi(g_1,g_2,g_3).$$
Let $g_i=(123)^{x_i}(13)^{y_i}(1234)^{z_i}$, $\sigma_i=(123)^{x_i}(13)^{y_i}$, and $h_i=(1234)^{z_i}$. 
The above lemma shows 
\begin{align*}
\lefteqn{\omega(g_1,g_2,g_3)} \\
 &=\alpha(x_1,(-1)^{y_1+z_1}x_2,(-1)^{y_1+y_2+z_1+z_2}x_3)(f(g_2\sigma_3)f(g_2)(-1)^{\epsilon(g_2)\epsilon(\sigma_3)})^{\epsilon(g_1)}.
\end{align*}

\subsection{Structure of $\C(S_4,\omega^l,H,1)$}
Let $H$ and $\omega$ be as in the previous subsection. 
We determine the structure of the group-theoretical fusion categories $\C_{1,l}=\C(S_4,\omega^l,H,1)$ for $l\in\{0,1,2,3,4,5\}$. 

Let $\gamma_0=e$, $\gamma_1=(12)(34)$, $\gamma_2=(123)$. 
Then we have the double coset decomposition
$$S_4=\bigsqcup_{i=0}^2 H\gamma_iH,$$
and
\begin{align*}
H\gamma_0H &= H,  \\
H\gamma_1H &=\gamma_1H= \{(13), (24), (12)(34), (14)(23)\}, \\
H\gamma_2H &= \{(1243),(1324),(1342),(1423),\\
    & \quad (123),(124),(132),(134),(142),(143),(234),(243),\\
    & \quad (12),(14),(23),(34) \}.
\end{align*}
The normalizer $N_{S_4}(H)$ of $H$ is 
$$H\rtimes_{-1} \langle\gamma_1 \rangle\cong \mathbb{Z}_4\rtimes_{-1}\mathbb{Z}_2\cong D_8.$$
and we have 
\[H^{\gamma_0}=H^{\gamma_1} = H, \quad H^{\gamma_2} = \{e\}.\]
Since  $\gamma_1=(12)(34)=(13)(1234) $ we can compute the Schur multiplier $\omega_{\gamma_1}\in Z^2(H^{\gamma_1},\mathbb{C}^\times)$ as  
\begin{align*}
\omega_{\gamma_1}(h_1,h_2)&=\omega(h_1,h_2,(13)(1234))=(f(h_2(13))f_0(h_2)(-1)^{\epsilon(h_2)})^{\epsilon(h_1)}  \\
 &=(f((13)h_2^{-1})f_0(h_2)(-1)^{\epsilon(h_2)})^{\epsilon(h_1)}=(f_0(h_2^{-1})f_0(h_2))^{\epsilon(h_1)}\\
 &=((-1)^{\epsilon(h_2^{-1})\epsilon(h_2)})^{\epsilon(h_1)}\\
 &=(-1)^{\epsilon(h_1)\epsilon(h_2)}=\frac{f_0(h_1)f_0(h_2)}{f_0(h_1h_2)},
\end{align*}
which shows $\omega_{\gamma_1}=df_0$. 
We let $\chi$ be one of the generators of $\widehat{H}\cong \mathbb{Z}_4$. 
Then the simple objects of $\C_{1,l}$ are 
$$X_{\gamma_0,1}, X_{\gamma_0,\chi},  X_{\gamma_0,\chi^2}, X_{\gamma_0,\chi^3},
X_{\gamma_1,f_0^l}, X_{\gamma_1,\chi f_0^l},X_{\gamma_1 \chi^2 f_0^l},X_{\gamma_1,\chi^3 f_0^l},X_{\gamma_2,1}.$$

Even though we do not know the fusion rules yet, by dimension considerations we see that the group of invertible objects $\Gamma(\C_{1,l})$ is  
\[
\{X_{\gamma_0,\chi^i}\}_{i=0}^3\cup\{X_{\gamma_1,\chi^i f_0^l}\}_{i=0}^3,
\]
and $X_{\gamma_2,1}$ with $\FPdim X_{\gamma_2,1} = 4$ is the only non-invertible object, hence corresponding to $\rho$.  
Thus $\C_{1,l}$ is a near-group category. 

\begin{thm}\label{basic} Let $\C_{1,l}=\C(S_4,\omega^l,H,1)$ be as above. 
\begin{itemize}
\item[$(1)$] For the second and third FS indicators of $X_{\gamma_2,1}$, we have 
$$\nu_2(X_{\gamma_2,1})=(-1)^l,$$
$$\nu_3(X_{\gamma_2,1})=2e^{\frac{-2\pi li}{3}}.$$
\item[$(2)$] 
The group $\Gamma(\C_{1,l})$ of the invertible objects in $\C_{1,l}$ is isomorphic to the dihedral group $D_8$ for even $l$, 
and the quaternion group $Q_8$ for odd $l$. 
\end{itemize}
In particular, the six group theoretical fusion categories $\C_{1,l}$, $l=0,1,2,3,4,5$, 
are pairwise inequivalent as pivotal fusion categories.
\end{thm}

\begin{proof} 
(1) Theorem \ref{fs-gt-form} implies 
$$\nu_k(X_{\gamma_2,1})= \sum_{\substack{r\in (123)H\\ r^k=e}} \pi_{-k}(r)= \sum_{\substack{r\in (123)H\\ r^k=e}}\prod_{j=1}^k \omega(r,r^{-j},r)^{-l}.$$
Since $(123)H = \{ (123), (1342), (243), (14) \}$, we have 
$$\nu_2(X_{(12),1})=\omega((14),(14),(14))^{-l}$$ and 
\begin{align*}
\lefteqn{\nu_3(X_{(12),1})} \\
 &=\omega((123),(123)^{-1},(123))^{-l}\omega((123),(123)^{-2},(123))^{-l}\\
 &+\omega((243),(243)^{-1},(243))^{-l}\omega((243),(243)^{-2},(243))^{-l} \\
 &=\omega((123),(123)^2,(123))^{-l}\omega((123),(123),(123))^{-l}\\
 &+\omega((243),(234),(243))^{-l}\omega((243),(243),(243))^{-l} \\
\end{align*}

Since $(14)=(123)(1234)^3$, we get 
\begin{align*}
\lefteqn{\omega((14),(14),(14))} \\
 &=\alpha(1,-1,1)f((14)(123))f((123)(1234)^3)=f((1234))f((1234)^3)=-1,
\end{align*}
and $\nu_2(X_{\gamma_2,1})=(-1)^l$. 

We have 
$$\omega((123),(132),(123))\omega((123),(123),(123))=\alpha(1,1,1)\alpha(1,2,1)=e^{\frac{2\pi i}{3}}.$$
Since $(243)=(123)(1234)^2$ and $(234)=(123)^2(13)(1234)$, we get 
\begin{align*}
\lefteqn{\omega((243),(234),(243))\omega((243),(243),(243))} \\
 &=\omega((123)(1234)^2,(123)^2(13)(1234),(123)(1234)^2)\\
 &\times \omega((123)(1234)^2,(123)(1234)^2,(123)(1234)^2) \\
 &=\alpha(1,2,1)\alpha(1,1,1) \\
 &=e^{\frac{2\pi i}{3}}.
\end{align*}
Thus we get  $\nu_3(X_{(12),1})=2e^{\frac{-2\pi l i}{3}}$.

(2) Since $N_{S_4}(H)=H\rtimes \langle \gamma_1\rangle$ and $\omega_{\gamma_1}=df_0$, we can identify 
$K$ in Theorem \ref{Inv} with $\{e,\gamma_1\}$, and get $\eta_{\gamma_1}(h)=f_0(h)^l$. 
Thus 
$$\nu(\gamma_1,\gamma_1)(h)=(f_0(h)f_0(h^{-1}))^l=\chi(h)^{2l}.$$
This means 
$$\Gamma(\C_{1,l})=\langle \chi,\widetilde{\gamma_1}|\;\chi^4=1,\; \widetilde{\gamma_1}^2=\chi^{2l},\; \widetilde{\gamma_1}\chi=\chi^{-1}\widetilde{\gamma_1}\rangle,$$ 
which is isomorphic to $D_8$ for even $l$, and to $Q_8$ for odd $l$. 
\end{proof}

\begin{rmk}
The category $\C_{1,0}$ is equivalent to the category of representations of the bismash product Hopf algebra 
$\mathbb{C}^{S_3} \# \mathbb{C}\mathbb{Z}_4$. This bismash product is the trivial case of a cleft extension Hopf algebra $\mathbb{C}^{S_3} \to H \to \mathbb{C}\mathbb{Z}_4$, analogous to the semidirect product of groups.
Such an extension gives a Singer pair structure for the Hopf algebras $(\mathbb{C}\mathbb{Z}_4, \mathbb{C}^{S_3})$. Equivalence classes of cleft extension Hopf algebras giving a fixed Singer pair structure form an abelian group denoted $\Opext(\mathbb{C}\mathbb{Z}_4, \mathbb{C}^{S_3})$. It was shown in \cite[Theorem 4.1]{M97} that $\Opext(\mathbb{C}\mathbb{Z}_4, \mathbb{C}^{S_3})=0$, hence this is the unique such Hopf algebra extension (up to equivalence of extensions) associated to the Singer pair $(\mathbb{C}\mathbb{Z}_4, \mathbb{C}^{S_3})$.

In a similar way, the other categories $\C_{1,l}$, $l\neq 0$ arise as the representation categories of cleft extension quasi-Hopf algebras; see  
\cite[Theorem 4.4]{N05} for an explicit construction. The group of equivalence classes of quasi-bialgebra extensions associated to the Singer pair $(\mathbb{C}\mathbb{Z}_4, \mathbb{C}^{S_3})$ is denoted $\Opext'(\mathbb{C}\mathbb{Z}_4, \mathbb{C}^{S_3})$. (See \cite{M02} for details.) 

The preceding results imply that $\mathbb{Z}_6$ is a subgroup of $\Opext'(\mathbb{C}\mathbb{Z}_4, \mathbb{C}^{S_3})$. Moreover, we may apply the Kac exact sequence associated with $\Opext'$ from \cite[Theorem 4.13]{M02} to see that
\[
0\cong H^2(\mathbb{Z}_4,\mathbb{C}^\times) \to \Opext'(\mathbb{C}\mathbb{Z}_4, \mathbb{C}^{S_3}) \to H^3(S_4, \mathbb{C}^\times) \to H^3(\mathbb{Z}_4, \mathbb{C}^\times) \cong \mathbb{Z}_4.
\]
Since $H^3(S_4, \mathbb{C}^\times) \cong \mathbb{Z}_6 \times \mathbb{Z}_4$ we have
\[
\Opext'(\mathbb{C}\mathbb{Z}_4, \mathbb{C}^{S_3}) \cong \mathbb{Z}_6
\]
\end{rmk}

\begin{rmk}
Dually, one may consider the group of equivalence classes of {\it coquasi}-bialgebra extensions associated to the Singer pair $(\mathbb{C}\mathbb{Z}_4, \mathbb{C}^{S_3})$, which is denoted $\Opext''(\mathbb{C}\mathbb{Z}_4, \mathbb{C}^{S_3})$.  It was shown in \cite[Example 5.2]{M03} that
\[
\Opext''(\mathbb{C}\mathbb{Z}_4, \mathbb{C}^{S_3})\cong \mathbb{Z}_4
\]
Thus we observe that
\[
\Opext'(\mathbb{C}\mathbb{Z}_4, \mathbb{C}^{S_3}) \oplus \Opext''(\mathbb{C}\mathbb{Z}_4, \mathbb{C}^{S_3}) \cong H^3(S_4, \mathbb{C}^\times),
\]
This observation indicates an interesting relationship between the groups $\Opext'$ and $\Opext''$ for other $\Opext$-trivial Singer pairs.
\end{rmk}

\section{General case}
We use the same notation as in the previous section. 

Recall that $S_3$ is isomorphic to $SL(2,\F_2)$, and $S_4=N\rtimes S_3$ is isomorphic to the affine group
$\F_2^2\rtimes SL(2,\F_2)$, which is identified with 
$$\left\{\left(
\begin{array}{ccc}
a &b &v  \\
c &d &w  \\
0 &0 &1 
\end{array}
\right)\Bigg|\;  v,w\in \F_2,\;\left(
\begin{array}{cc}
a &b  \\
c &d 
\end{array}
\right)\in SL(2,\F_2)\right\}.$$
This expression of $S_4$ suggests possible generalization of our previous construction. 
Let 
\[G_n = \left\{ \left(\begin{matrix}
  a & b & v \\
  c & d & w \\
  0_n^T & 0_n^T & I_n
 \end{matrix}\right) \in SL_{n+2}(\mathbb{F}_2) \;  \Bigg| \; v, w \in \mathbb{F}_2^n,\; \left(\begin{matrix}
 a & b \\
 c & d \\
 \end{matrix}\right)\in SL_2(\mathbb{F}_2) \right\} 
\]
where we let $0_n=(0,0,\ldots, 0)\in \mathbb{F}_2^n$ and $I_n\in SL_n(\mathbb{F}_2)$ is the identity matrix.
Then $G_n\cong \F_2^{2n}\rtimes S_3$, and again we can inflate the 3-cocycle $\omega_0\in Z^3(S_3,\mathbb{C}^\times)$ 
to $G_n$. 
The first named author conjectured in \cite{I17} that with an appropriate choice of a subgroup $H_n\leq G_n$, the same construction 
as in the previous section, using $G_n$ instead of $S_4$, exhausts general noncommutative near-group categories, 
which we are going to prove now. 
Note that $G_n$ is a semi-direct product $\F_2^{2n-2}\rtimes S_4$, where the action of $S_4$ on $\F_2^{2n-2}$ 
is given through $\pi:S_4\to S_3\cong SL(2,\F_2)$. 
For our purpose, it is more convenient to directly work on the latter expression as we already have an adapted cocycle $\omega$ of $S_4$ 
for the subgroup $H$. 
For this reason, we start with our argument by giving an alternative definition of $G_n$. 

Let $V_n=\F_2^{n-1}$.  
We define an action of $S_3$ on $V_n\oplus V_n$ as follows:
\begin{align*}
(12)\cdot (v,w)&=(w,v),\\ 
(13)\cdot(v,w)&=(v+w,w),\\
(23)\cdot(v,w)&=(v,v+w),\\
(123)\cdot (v,w)&=(v+w,v),\\ 
(132)\cdot (v,w)&=(w,v+w),
\end{align*}
and extend it to an $S_4$-action through $\pi$. 
We set $G_n=(V_n\oplus V_n)\rtimes S_4$, which is $(V_n\oplus V_n)\times S_4$ as a set with multiplication 
$$((v_1,v_2),g)((w_1,w_2),h)=((v_1,v_2)+g\cdot (w_1,w_2),gh).$$
We denote by $p$ the projection $p:G_n\to S_4$ onto the second components. 
For simplicity, we denote $((v_1,v_2),e)$ by $(v_1,v_2)$ and $((0,0),g)$ by $g$. 

Note that $(v,0_{n-1})$ commutes with $h$ for every $v\in V_n$ and $h\in H$ as $(13)\cdot(v,0_{n-1})=(v,0_{n-1})$.  
Let $\omega_n(x,y,z)=\omega(p(x),p(y),p(z))$, and let \[
H_n=(V_n\oplus \{0_{n-1}\}) H\cong \mathbb{Z}_2^{n-1} \times \mathbb{Z}_4.
\]
Then $\omega_n$ is an adapted 3-cocycle of $G_n$ for $H_n$. 
We then define 
\[
\C_{n,l}=\C(G_n,\omega_n^l,H_n,1)\]
for $l=0,1,2,3,4,5$. 

Note that $\pi(\gamma_1)=e$, and $\gamma_1$ commutes with $(v_1, v_2)$ for all $v_1,v_2\in V_n$. 
Thus $K_n:=(\{0_{n-1}\}\oplus V_n)\{e,\gamma_1\}$ is a subgroup of $G_n$ isomorphic to $\mathbb{Z}_2^n$, and we have 
$N_{G_n}(H_n)=(V_n\oplus V_n)N_{S_4}(H)=H_n\rtimes K_n$. 
We have $(\omega_n)_{(0_{n-1},v)}=1$ and 
$$(\omega_n)_{(0_{n-1},v)\gamma_1}(h_1,h_2)=\omega_{\gamma_1}(p(h_1),p(h_2))=d((f_0\circ p)^l)(h_1,h_2)$$ 
for $h_1,h_2\in H_n$.  

We claim 
$$G_n=H_n\gamma_2H_n\sqcup \bigsqcup_{k\in K_n}H_nkH_n.$$
Indeed, direct computation shows $(H_n)^{\gamma_2}=\gamma_2H_n\gamma_2^{-1}\cap H_n=\{e\}$, which implies $|H_n\gamma_2H_n|=|H_n|^2=2^{2n+2}$. 
Since $H_nkH_n=kH_n$ for $k\in K_n$, the cardinality of the right-hand side is $2^{2n+2}+|H_n||K_n|=2^{2n+1}\cdot3$, which coincides with $|G_n|$. 
From the claim, we see that 
$$\Gamma(\C_{n,l})=\{X_{(0_{n-1},v),\tau}\}_{\tau\in \widehat{H_n},\;v\in V_n}\sqcup\{X_{(0_{n-1},v)\gamma_1,\tau f_0^l\circ p}\}_{\tau\in \widehat{H_n},\;v\in V_n}$$
and $X_{\gamma_2,1}$ is the only non-invertible simple object. 
Thus $\C_{n,l}$ is a near-group category. 
We have $\FPdim X_{\gamma_2,1}=2^{n+1}$.

\begin{thm}\label{general} Let $\C_{n,l}=\C(G_n,\omega_n^l,H_n,1)$ be as above. 
\begin{itemize}
\item[$(1)$] For the second and third FS indicators of $X_{\gamma_2,1}$, we have 
$$\nu_2(X_{\gamma_2,1})=(-1)^l,$$
$$\nu_3(X_{\gamma_2,1})=2^ne^{\frac{-2\pi li}{3}}.$$
\item[$(2)$] 
For even $l$, $\Gamma(\C_{n,l})$ is isomorphic to the central product of $n$ copies of $D_8$.  
For odd $l$, it is isomorphic to the central product of $Q_8$ and $n-1$ copies of $D_8$. 
\end{itemize}
In particular, the six group theoretical fusion categories $\C_{n,l}$, $l=0,1,2,3,4,5$, 
are pairwise inequivalent as pivotal fusion categories.
\end{thm}

\begin{proof} 
(1) We have 
\begin{align*}
\lefteqn{(123)H_n} \\
 &=(123)\{(v,0_{n-1})h\}_{v\in V_n,\;h\in H} \\
 &= \{(v,v)(123)h\}_{v\in V_n,\;h\in H}\\
 &=\bigcup_{v\in V_n}\{(v,v)(123), (v,v)(1342),(v,v) (243),(v,v) (14) \}.
\end{align*}
Since $\pi((14))=(23)$, we have $((v,v)(14))^2=(v,0_{n-1})$, 
and $(14)$ is the only order 2 element in $(123)H_n$. 
Thus the same computation as in the proof of Theorem \ref{basic} shows $\nu_2(X_{\gamma_2,1})=(-1)^l$. 

Since $\pi((123))=\pi((243))=(123)$, we see that the set of order 3 elements in $(123)H_n$ is 
$$\bigcup_{v\in V_n}\{(v,v)(123),(v,v) (243) \}.$$
Thus the same computation as in the proof of Theorem \ref{basic} shows $\nu_3(X_{(12),1})=2^ne^{\frac{-2\pi l i}{3}}$.

(2) 
Recall that $\chi$ is one of the generator of $\widehat{H}$. 
We regard $\chi$ as an element of $\widehat{H_n}$ by setting $\chi((v,0_{n-1})h)=\chi(h)$.  
For $v^*\in V_n^*$, we define $\mu_{v^*}\in \widehat{H_n}$ as $\mu_{v^*}((v,0_{n-1})h)=(-1)^{\langle v,v^*\rangle}$. 

We can identify $K_n$ with $K$ in Theorem \ref{Inv}, and get $\eta_{(0_{n-1},w)}=1$ and $\eta_{(0_{n-1},w)\gamma_1}(h)=f_0(\rho(h))^l$. 
Thus 
$$\nu((0_{n-1},w_1)\gamma_1^r,(0_{n-1},w_2)\gamma_1^s)=\left\{
\begin{array}{ll}
 1, &\quad (r,s)\neq (1,1) \\
 \chi^{2l}, &\quad (r,s)=(1,1).
\end{array}
\right.$$
The group $\Gamma(\C_{1,l})$ is identified with a group generated by $\widehat{H_n}$ and symbols $\{\widetilde{k}\}_{k\in K_n}$ with the relations 
$\widetilde{k_1}\widetilde{k_2}=\nu(k_1,k_2)\widetilde{k_1k_2}$ and $\widetilde{k}\tau={}^k\tau\widetilde{k}$ 
for all $k_1,k_2,k\in K_n$ and $\tau\in \widehat{H_n}$. 
Since
\begin{align*}
\lefteqn{{}^{(0_{n-1},w)}\tau((v,0_{n-1})(1234)^z)} \\
 &=\tau((0_{n-1},-w)(v,0_{n-1})(1234)^z(0_{n-1},w))=\tau((v+zw,0_{n-1})(1234)^z), 
\end{align*}
and ${}^{\gamma_1}\tau((v,0_{n-1})h)=\tau((v,0_{n-1})h^{-1})$, 
we have ${}^{(0_{n-1},w)}\chi=\chi$, ${}^{(0_{n-1},w)}\mu_{v^*}=\mu_{v^*}(\chi^2)^{\langle w,v^*\rangle}$, ${}^{\gamma_1}\chi=\chi^{-1}$, and 
${}^{\gamma_1}\mu_{v^*}=\mu_{v^*}$. 

Let $\{e_i\}_{i=1}^{n-1}$ be the canonical basis of $V_n$ and let $\{e_i^*\}_{i=1}^{n-1}$ be its dual basis of $V_n^*$. 
We set $\mu_i=\mu_{e_i^*}$ and $f_i=\widetilde{(0_{n-1},e_i)}$. 
We set $c=\chi^2$, which is a central element in $\Gamma(\C_{1,l})$ of order 2. 
Then $\Gamma(\C_{1,l})$ is generated by 
$$\{\mu_i\}_{i=1}^{n-1}\cup \{\chi\}\cup\{f_i\}_{i=1}^{n-1}\cup\{\widetilde{\gamma_1}\}.$$
Let $\Gamma_i$ be the subgroup of $\Gamma(\C_{1,l})$ generated by $\mu_i$ and $f_i$ for $1\leq i\leq n-1$, and 
let $\Gamma_n$ be the subgroup generated by $\chi$ and $\widetilde{\gamma_1}$. 
Then $\{\Gamma_i\}_{i=1}^n$ are mutually commuting family of subgroups with common center $\{e,c\}$. 
As before $\Gamma_n$ is isomorphic to $D_8$ for even $l$, and to $Q_8$ for odd $l$. 
Since $f_i\mu_i=c\mu_if_i$ and $\mu_i^2=f_i^2=e$, the subgroup $\Gamma_i$ for $1\leq i\leq n-1$ is isomorphic to $D_8$.  
Since $|\Gamma(\C_{1,l})|=2^{2n+1}$, we see that $\Gamma(\C_{1,l})$ is the central product of $\{\Gamma_i\}_{i=1}^n$. 
\end{proof}

\begin{rmk} 
It might be interesting to replace $\F_2$ with $\F_4$ in our argument and investigate the resulting fusion categories. 
In the simplest case  
$$G=\F_4^2\rtimes SL(2,\F_4)=\left\{\left(
\begin{array}{ccc}
a &b &v  \\
c &d &w  \\
0 &0 &1 
\end{array}
\right)\Bigg|\;  v,w\in \F_4,\;\left(
\begin{array}{cc}
a &b  \\
c &d 
\end{array}
\right)\in SL(2,\F_4)\right\},$$ 
$$H=\left\{\left(
\begin{array}{ccc}
1 &v &w  \\
0 &1&v  \\
0 &0 &1 
\end{array}
\right)\in G\Bigg|\;  v,w\in \F_4,\right\},$$
we get a matched pair $(SL(2,\F_4),H)$, and $\C(G,1,H,1)$ is a representation category of 
the Hopf algebra $\mathbb{C}^{A_5} \# \mathbb{C}H$, which is not a near-group category but a quadratic category. 
The class of quadratic categories include near-group categories, and it draws attention of specialists recently 
(see for example \cite{I18}).  
\end{rmk}


\end{document}